\documentclass[12pt]{article}
\usepackage{amsmath,amssymb,amsbsy,amsfonts,amsthm,latexsym,graphicx,color,makeidx,
         amsopn,amstext,amsxtra,euscript,amscd,amsthm,mathtools}
\usepackage[makeroom]{cancel}
\theoremstyle{plain}

\numberwithin{equation}{section}

\allowdisplaybreaks

\newtheorem{lem}{Lemma}

\newtheorem{prop}{Proposition}

\def\N{{\mathbb N}}
\def\Z{{\mathbb Z}}

\def\dat{6 avril 2020}

\begin{document}

\centerline{\bf Bounds for the counting function of the Jordan-P\'olya numbers}
\vskip 10pt

\begin{center}
Jean-Marie De Koninck \\
(corresponding author) \\
D\'epartement de math\'ematiques et de statistique\\
Universit\'e Laval\\
Qu\'ebec G1V 0A6, Canada \\
jmdk@mat.ulaval.ca
\vskip 10pt

Nicolas Doyon\\
D\'epartement de math\'ematiques et de statistique\\
Universit\'e Laval\\
Qu\'ebec G1V 0A6, Canada\\
nicolas.doyon@mat.ulaval.ca
\vskip 10pt

A. Arthur Bonkli Razafindrasoanaivolala \\
D\'epartement de math\'ematiques et de statistique\\
Universit\'e Laval\\
Qu\'ebec G1V 0A6, Canada\\
arthur@aims.edu.gh
\vskip 10pt

William Verreault \\
D\'epartement de math\'ematiques et de statistique\\
Universit\'e Laval\\
Qu\'ebec G1V 0A6, Canada \\
william.verreault.2@ulaval.ca
\end{center}
\vskip 10pt

\hfill {\it \'Edition du \dat}
\vskip 15pt

\begin{abstract}
A positive integer $n$ is said to be a {\it Jordan-P\'olya number} if it can be written as a product of factorials. We obtain non-trivial lower and upper bounds for the number of Jordan-P\'olya numbers not exceeding a given number $x$.
\end{abstract}
\vskip 5pt

\vskip 10pt

\section{Introduction}

A positive integer $n$ is said to be a {\it Jordan-P\'olya number} if it can be written as a product of factorials.
Jordan-P\'olya numbers arise naturally in a simple combinatorial problem. Given  $k$  groups of  $n_1,n_2,\ldots,n_k$  distinct objects, then the number of distinct permutations of these  $n_1+n_2+\cdots+n_k$  objects which maintain objects of the same group adjacent is equal to  $k!\cdot n_1!\cdot n_2!\cdots n_k!$, a Jordan-P\'olya number.
\vskip 5pt

The Jordan-P\'olya numbers below 10,000 are
\vskip 5pt

\noindent
1, 2, 4, 6, 8, 12, 16, 24, 32, 36, 48, 64, 72, 96, 120, 128, 144,
192, 216, 240, 256, 288, 384, 432, 480, 512, 576, 720, 768, 864, 960,
1024, 1152, 1296, 1440, 1536, 1728, 1920, 2048, 2304, 2592, 2880,
3072, 3456, 3840, 4096, 4320, 4608, 5040, 5184, 5760, 6144, 6912,
7680, 7776, 8192, 8640, 9216.
\vskip 5pt

For a longer list, see the {\it On-Line Encyclopedia of Integer Sequences}, Sequence A001013.
Much study has been done on a particular subset of the Jordan-P\'olya numbers, namely those  which are themselves factorials. In particular, consider the equation
\begin{equation} \label{eq:1-1}
n! = a_1! a_2! \cdots a_r! \quad \mbox{ in integers } n>a_1 \ge a_2 \ge \cdots \ge a_r \ge 2,\ r\ge 2.
\end{equation}
This equation has infinitely many ``trivial'' solutions. Indeed, choose any integers $a_2 \ge \cdots \ge a_r \ge 2$ and set $n=a_2!\cdots a_r!$. Then, choose  $a_1=n-1$. One can easily see that $n! = n\cdot (n-1)! = a_1! a_2! \cdots a_r!$. Besides these trivial solutions of equation (\ref{eq:1-1}), we find the non-trivial solutions
\begin{equation} \label{eq:1-2}
9!=2!\cdot 3!^2 \cdot 7!, \qquad 10!=6!\cdot 7!=3!\cdot 5! \cdot 7!, \qquad 16!=2!\cdot 5! \cdot 14!.
\end{equation}
According to Hickerson's conjecture, there are no other non-trivial solutions for equation (\ref{eq:1-1}).
In 2007, Luca \cite{kn:luca} showed that if the $abc$ conjecture holds, then equation (\ref{eq:1-1}) has only a finite number of non-trivial solutions.
In 2016, Nair and Shorey \cite{kn:nair-shorey} showed that any other non-trivial solution $n$ of (\ref{eq:1-1}), besides those in (\ref{eq:1-2}),
must satisfy $n>e^{80}$.
\vskip 5pt

On the other hand, more than 40 years ago, Erd\H{o}s and Graham \cite{kn:erdos-graham} showed that the number of distinct integers of the form $a_1! a_2! \cdots a_r!$, where $a_1<a_2<\cdots <a_r \le y$ is $\exp\{ (1+o(1)) y (\log \log y)/\log y  \}$ as $y\to \infty$.
\vskip 5pt

Here, letting ${\cal J}$ stand for the set of Jordan-P\'olya numbers and  ${\cal J}(x)$ for its counting function, we show that ${\cal J}(x) =o(x)$ and in fact, given any small $\varepsilon>0$, we show the much stronger estimate
\begin{equation} \label{eq:1}
\mathcal{J}(x) < \exp\left\{ (4+\varepsilon)\frac{\sqrt{\log x}\log\log\log x}{\log\log x} \right\} \qquad (x\ge x_1)
\end{equation}
for some $x_1=x_1(\varepsilon)>0$.
We also show that, for any given $\varepsilon>0$, there exists $x_2=x_2(\varepsilon)$ such that
\begin{equation} \label{eq:1a}
{\cal J}(x) > \exp\left\{ (2-\varepsilon) \frac{\sqrt{\log x}}{\log \log x} \right\} \qquad (x\ge x_2).
\end{equation}

\section{Preliminary results}

We first mention some known results in the form of lemmas and propositions that will prove useful in establishing the lower and upper bounds for ${\cal J}(x)$.

We start with a weak form of Stirling's formula for the factorial function, a proof of which can be found on page 11 in the book of De Koninck and Luca \cite{kn:dk-fl}.
\begin{lem} \label{lem:weak}
For each integer $m\ge 1$, we have
$$m! > \left( \frac me \right)^m.$$
\end{lem}

We now state a more precise form of Stirling's formula, which is a particular case of formula (4) in a 2009 paper of De Angelis \cite{kn:angelis}.
\begin{lem} \label{lem:stirling}
For all  integers $n\ge 2$,
$$n!= \left(\frac ne\right)^n \sqrt{2\pi n} \left( 1 + \frac 1{12n} + O \left( \frac 1{n^2} \right) \right).$$
\end{lem}

\begin{lem} \label{lem:binomial}
Given any positive integers $a$ and $b$,
$$
{ a+b\choose a}\le\left(\frac{e(a+b)}{a}\right)^a.
$$
\end{lem}

\begin{proof}
This follows from the following string of inequalities.
$${ a+b\choose a}=\frac{b+1}{1} \times \frac{b+2}{2}\times \cdots \times \frac{b+a}{a}\le \frac{(b+a)^a}{a!}\le\left(\frac{e(a+b)}{a}\right)^a,$$
where we used Lemma \ref{lem:weak} for this last inequality.
\end{proof}

\begin{lem} \label{lem:feller}
Given positive integers $k\le R$, the number $S_k(R)$ of solutions $(r_1,r_2,\ldots,r_k)$ in non-negative integers $r_1,r_2,\ldots,r_k$ of the inequality
$$r_1+r_2 + \cdots + r_k \le R$$
satisfies  $\displaystyle{S_k(R)= {R+k \choose R} }$.
\end{lem}

\begin{proof}
It follows from formula (5.2) in the book of W. Feller \cite{kn:feller} that the number of ways of writing a positive integer $m$ as a sum of $k$ non-negative integers is equal to ${m+k-1 \choose k-1}$. Therefore, since $S_k(R)$ is the sum of this last expression as $m$ varies from 0 to $R$,  we find, using induction, that
$$S_k(R) = \sum_{m=0}^R   {m+k-1 \choose k-1} = {R+k \choose k} = {R+k \choose R}.$$
\end{proof}

The next result, which is of independent interest, is a key element in the proof of the upper bound for ${\cal J}(x)$. Essentially, it says that the sequence of the exponents in the prime factorisation of $m!$ decreases faster than the sequence of the primes to which they are attached increases.
\begin{lem} \label{lem:binomial-prime}
Let the prime factorisation of $m!$ be written as
$$m! = 2^{\alpha_2} \cdot 3^{\alpha_3}  \cdot 5^{\alpha_5} \cdots  p_t^{\alpha_{p_t}},$$
where $p_t$ is the largest prime number not exceeding $m$. Then, given any primes $p,q$ such that $p<q\le p_t$, we have
$$\frac{\alpha_p}{\alpha_q} \ge \left\lfloor \frac qp \right\rfloor.$$
\end{lem}

\begin{proof}
Let $p<q\le p_t$ be fixed. Then, there exist two positive integers $u\ge v$ such that
\begin{eqnarray} \label{eq:light-1}
\alpha_p & = & \left\lfloor \frac mp \right\rfloor + \left\lfloor \frac m{p^2} \right\rfloor + \cdots + \left\lfloor \frac m{p^u} \right\rfloor,\\ \label{eq:light-2}
\alpha_q & = & \left\lfloor \frac mq \right\rfloor + \left\lfloor \frac m{q^2} \right\rfloor + \cdots + \left\lfloor \frac m{q^v} \right\rfloor.
\end{eqnarray}
Let $k$ be the unique positive integer satisfying $kp < q < (k+1)p$. Clearly, our claim will be proved if we can show that
\begin{equation} \label{eq:claim-1}
\alpha_p \ge k\, \alpha_q.
\end{equation}
Now, if we can show that
\begin{equation} \label{eq:claim-2}
\left\lfloor \frac mp \right\rfloor \ge k \left\lfloor \frac mq \right\rfloor,
\end{equation}
then surely we will have
$\displaystyle{\left\lfloor \frac m{p^i} \right\rfloor \ge k \left\lfloor \frac m{q^i} \right\rfloor}$ for each $i=2,3,\ldots,u$ and therefore, in light of (\ref{eq:light-1}) and (\ref{eq:light-2}), inequality (\ref{eq:claim-1}) will follow.
This means that we only need to prove (\ref{eq:claim-2}). Now, there exist two positive integers $r_1$ and $r_2$ such that
\begin{eqnarray*}
m & = & r_1p + \theta_1\ \mbox{ for some non-negative integer }\theta_1 \le p-1,\\
m & = & r_2q + \theta_2\ \mbox{ for some non-negative integer }\theta_2 \le q-1,
\end{eqnarray*}
and therefore,
$$1 = \frac{r_1p + \theta_1}{r_2q + \theta_2} \le \frac{r_1p + p-1}{r_2q} < \frac{r_1p + p-1}{r_2\cdot kp},$$
thereby establishing that
$$r_1p + p-1 > k r_2 p,$$
so that
\begin{equation} \label{eq:follow}
r_1 + \frac{p-1}p > kr_2.
\end{equation}
Since $r_1$ and $r_2$ are two integers whereas $\displaystyle{\frac{p-1}p}$ is a positive number smaller than 1, it follows from (\ref{eq:follow}) that $r_1\ge kr_2$, which proves (\ref{eq:claim-2}) since $\displaystyle{r_1= \left\lfloor \frac mp \right\rfloor}$
and $\displaystyle{r_2= \left\lfloor \frac mq \right\rfloor}$.
\end{proof}

The following result provides very useful explicit upper and lower bounds for the $k$-th prime number. 
\begin{lem} \label{lem:rosser}
If $p_k$ stands for the $k$-th prime number, then 
\begin{equation} \label{eq:pk-6}
p_k < k \log k + k \log \log k \qquad (k\ge 6)
\end{equation}
and 
\begin{equation} \label{eq:pk}
p_k>k\log k \qquad (k \ge 1).
\end{equation}
\end{lem}

\begin{proof}
The first inequality is due to Rosser \cite{kn:rosser}, whereas the second is due to Rosser and Schoenfeld  \cite{kn:scho}.
\end{proof}

The prime number theorem can be written in various forms. We will be using the following, which is essentially Theorem 5.1 in the book of Ellison and Ellison \cite{kn:ellison}.
\begin{prop} \label{prop:pnt}
Setting $\theta(x):=\sum_{p\le x} \log p$, there exists an absolute constant $a>0$ such that
$$\theta(x) = x\left( 1 + O \left( \frac 1{e^{a\sqrt{\log x}}}\right) \right).$$
\end{prop}

Let $\Psi(x,y):=\#\{n\le x: P(n)\le y\}$, where $P(n)$ stands for the largest prime factor of $n\ge 2$, with $P(1)=1$. Moreover, let $\pi(x)$ stand for the number of prime numbers not exceeding $x$. The following estimate can be found in Granville \cite{kn:granville}.
\begin{prop} \label{prop:granville}
If $y=o(\log x)$ as $x\to \infty$, then
$$
\Psi(x,y) = \left( \frac{\log x}y \right)^{(1+o(1))\pi(y)}.
$$
\end{prop}

The following is a 1969 result of  Ennola \cite{kn:ennola}, a proof of which is given in the book of Tenenbaum \cite{kn:ten}.
\begin{prop} \label{prop:ennola}
Let $a_1,a_2,\ldots$ be a sequence of positive real numbers and set
$$N_k(z):= \#\left\{ (\nu_1,\nu_2,\ldots,\nu_k)\in \Z^k : \nu_1\ge 0,\ldots,\nu_k\ge 0,\ \sum_{i=1}^k \nu_i a_i \le z  \right\}.$$
Then, for each positive integer $k$,
\begin{equation} \label{eq:ennola}
\frac{z^k}{k!} \prod_{i=1}^k \frac 1{a_i} < N_k(z) \le \frac{(z+\sum_{i=1}^k a_i  )^k}{k!} \prod_{i=1}^k \frac 1{a_i}.
\end{equation}
\end{prop}

\section{The proof of the upper bound}

Observe that for every integer $n$ counted by ${\cal J}(x)$, each of its prime factors must be smaller than $\displaystyle{2\frac{\log x}{\log \log x}}$, provided $x$ is sufficiently large. We now define the four integers $r,r_1,r_2,r_3$ each depending on $x$ as follows.
\begin{eqnarray*}
r & = & \pi\left(2\frac{\log x}{\log \log x}\right), \mbox{ so that }\
r\le 3\frac{\log x}{(\log\log x)^2}, \\
r_1 & = & \pi\left(\frac{\sqrt{\log x}}{\log\log x}\right), \mbox{ which is asymptotic to } \frac{2\sqrt{\log x}}{(\log\log x)^2}\mbox{ as } x\to \infty, \\
r_2 & = & \pi\left(\sqrt{\log x}\right), \mbox{ which is asymptotic to } \frac{2\sqrt{\log x}}{\log\log x}\mbox{ as } x\to \infty, \\
r_3 & = & \pi\left(\sqrt{\log x}\log\log x\right), \mbox{ which is asymptotic to } 2\sqrt{\log x}\mbox{ as } x\to \infty.
\end{eqnarray*}
Let $m$ be a positive integer and $q_1,q_2$ be two prime numbers such that $q_1<q_2\le m$.  Assuming that $q_1^{\eta_1}\|m!$ and that $q_2^{\eta_2}\|m!$, it follows from Lemma \ref{lem:binomial-prime} that
$$
\frac{\eta_1}{\eta_2}\ge   \left\lfloor \frac{q_2}{q_1}\right\rfloor.
$$
Using these observations, we may write that $\mathcal{J}(x)\le \#\mathcal{A}(x)$, where
$$
\mathcal{A}(x): =\left\{a=(a_1,a_2,\ldots, a_r)\in\N^r:\sum_{j=1}^r a_j\log p_j\le \log x, \frac{a_i}{a_j}\ge \left\lfloor \frac{p_j}{p_i}\right\rfloor \right\}.
$$
In order to derive an upper bound for $\#\mathcal{A}(x)$, we introduce the four sets
\begin{eqnarray*}
\mathcal{B}_1(x) &:= &\{(b_1,\ldots,b_{r_1}):\exists a\in \mathcal{A}(x), a_1=b_1,\ldots, a_{r_1}=b_{r_1}\},\\
\mathcal{B}_2(x) & := & \{(b_1,\ldots, b_{r_2-r_1}):\exists a\in \mathcal{A}(x), a_{r_1+1}=b_1,\ldots, a_{r_2}=b_{r_2-r_1}\},\\
\mathcal{B}_3(x) & := & \{(b_1,\ldots, b_{r_3-r_2}):\exists a\in \mathcal{A}(x), a_{r_2+1}=b_1,\ldots, a_{r_3}=b_{r_3-r_2}\},\\
\mathcal{B}_4(x) & := & \{(b_1,\ldots, b_{r-r_3}):\exists a\in \mathcal{A}(x), a_{r_3+1}=b_1,\ldots, a_{r}=b_{r-r_3}\}.
\end{eqnarray*}
It is then clear that
$$\#\mathcal{A}(x)\le \#\mathcal{B}_1(x)\times \#\mathcal{B}_2(x) \times \#\mathcal{B}_3(x)\times \#\mathcal{B}_4(x).$$
We will now provide upper bounds for each of the quantities $\#\mathcal{B}_j(x)$ for $1\le j\le 4$.

Let $\varepsilon>0$ be an arbitrarily small number and let $x$ be a large number.

First observe that
$$
\#\mathcal{B}_1(x)\le \# \left\{(b_1,b_2,\ldots, b_{r_1}):\sum_{j=1}^{r_1} b_j\log p_j\le \log x \right\}.
$$
From this, it follows from Proposition \ref{prop:granville} that
$$
\#\mathcal{B}_1(x)\le \Psi\left(x,p_{r_1}\right) = \left(\frac{\log x}{p_{r_1}}\right)^{(1+o(1))r_1}\le (\sqrt{\log x}\log\log x)^{3\sqrt{\log x}/(\log\log x)^2},
$$
so that
\begin{equation}\label{b1}
\#\mathcal{B}_1(x)\le \exp\left(2\frac{\sqrt{\log x}}{\log\log x}\right).
\end{equation}

On the other hand, we have
\begin{eqnarray}\label{aux} \nonumber
\#\mathcal{B}_2(x) & \le & \#\left\{(b_1,\ldots,b_{r_2-r_1}):\left(\sum_{j=1}^{r_2-r_1}b_j\log p_{j+r_1}\right)+\left(b_1\sum_{j=1}^{r_1}\left\lfloor\frac{p_{r_1}}{p_j}\right\rfloor\log p_j\right)\le \log x\right\} \\
& \le & \#\left\{(b_1,\ldots,b_{r_2-r_1}):b_1\sum_{j=1}^{r_1}\left\lfloor\frac{p_{r_1}}{p_j}\right\rfloor\log p_j \le \log x\right\},
\end{eqnarray}
where we used the fact guaranteed by Lemma \ref{lem:binomial-prime}  that $\displaystyle{\frac{a_j}{b_1}\ge \left\lfloor\frac{p_{r_1}}{p_j}\right\rfloor}$ for $1\le j\le r_1$. We then perform a change of variable, namely the one given implicitly by
$$
b_k=\sum_{j=k}^{r_2-r_1}c_j,\quad 1\le k\le r_2-r_1.
$$
Given that the sequence $b_1,b_2,\ldots, b_{r_2-r_1}$ is non-increasing, we have $c_j\ge 0,\, 1\le j\le r_2-r_1$. From (\ref{aux}), it follows that
\begin{equation}\label{aux2}
\#\mathcal{B}_2(x)\le\#\left\{(c_1,\ldots,c_{r_2-r_1}):\left(\sum_{j=1}^{r_2-r_1}c_j\right)\left(\sum_{j=1}^{r_1}\left\lfloor\frac{p_{r_1}}{p_j}\right\rfloor\log p_j\right)\le \log x\right\}.
\end{equation}

Now,  it follows from  the prime number theorem that
\begin{equation} \label{eq:double}
\sum_{p\le y} \frac{\log p}p >(1-\varepsilon)\log y,
\end{equation}
provided  $y$ is sufficiently large.

Using inequality (\ref{eq:pk}) of Lemma \ref{lem:rosser}, as well as inequality (\ref{eq:double}) and Proposition \ref{prop:pnt}, we may write that
\begin{eqnarray} \label{eq:before} \nonumber
\sum_{j=1}^{r_1}\left\lfloor\frac{p_{r_1}}{p_j}\right\rfloor\log p_j & \ge &
p_{r_1} \sum_{j=1}^{r_1} \frac{\log p_j}{p_j} - \sum_{j=1}^{r_1}\log p_j \\ \nonumber
& \ge & (1-\varepsilon)
 \frac{\sqrt{\log x}}{\log \log x}\sum_{j=1}^{r_1}\frac{\log p_j}{p_j}-  \theta(p_{r_1}) \\ \nonumber
& \ge & (1-\varepsilon)
 \frac{\sqrt{\log x}}{\log \log x}(1-\varepsilon) \log p_{r_1} - (1+\varepsilon) p_{r_1} \\
& \ge &  \frac 13 \sqrt{\log x}.
\end{eqnarray}
Using this in (\ref{aux2}), we get
$$
\#\mathcal{B}_2(x)\le \#\left\{(c_1,\ldots, c_{r_2-r_1}):\sum_{j=1}^{r_2-r_1}c_j\le 3\sqrt{\log x}\right\},
$$
which, in light of Lemma \ref{lem:feller},  yields
\begin{eqnarray*}
\#\mathcal{B}_2(x) & \le & {r_2-r_1+\lceil3\sqrt{\log x}\rceil \choose \lceil 3\sqrt{\log x}\rceil}
\le \binom{\lceil\frac{2\sqrt{\log x}}{\log\log x}\rceil+\lceil 3\sqrt{\log x}\rceil}{\lceil 3\sqrt{\log x}\rceil}\\
& = & \binom{\lceil\frac{\sqrt{\log x}}{\log\log x}\rceil+\lceil 3\sqrt{\log x}\rceil}{\lceil\frac{\sqrt{\log x}}{\log\log x}\rceil}
\le \binom{2\lceil 3\sqrt{\log x}\rceil}{\lceil\frac{\sqrt{\log x}}{\log\log x}\rceil},
\end{eqnarray*}
where we used the fact that for any positive integers $a$ and $b$, we have $\displaystyle{ {a+b \choose b} = {a+b \choose a}}$.
Using Lemma \ref{lem:binomial}, it then follows that
\begin{equation}\label{b2}
\#\mathcal{B}_2(x)\le \exp\left\{\frac{\sqrt{\log x}\log\log\log x}{\log\log x}\right\}.
\end{equation}

An upper bound for $\#\mathcal{B}_3(x)$ is obtained using a similar technique. We have
\begin{equation}\label{aux3}
\#\mathcal{B}_3(x)\le\#\left\{(b_1,\ldots,b_{r_3-r_2}):\left(\sum_{j=1}^{r_3-r_2}b_j\log p_{j+r_2}\right)+\left(b_1\sum_{j=1}^{r_2}\left\lfloor\frac{p_{r_2}}{p_j}\right\rfloor\log p_j\right)\le \log x\right\}.
\end{equation}
Performing the change of variable
$$
b_k=\sum_{j=k}^{r_3-r_2}c_j,\quad 1\le k\le r_3-r_2,
$$
we obtain from (\ref{aux3}) that
\begin{equation}\label{aux4}
\#\mathcal{B}_3(x)\le\#\left\{(c_1,\ldots,c_{r_3-r_2}):\left(\sum_{j=1}^{r_3-r_2}c_j\right)\left(\sum_{j=1}^{r_2}\left\lfloor\frac{p_{r_2}}{p_j}\right\rfloor\log p_j\right)\le \log x\right\}.
\end{equation}

Again using (\ref{eq:pk}), (\ref{eq:double}) and  Proposition \ref{prop:pnt}, while proceeding as we did to obtain (\ref{eq:before}), we find that
$$
\sum_{j=1}^{r_2}\left\lfloor\frac{p_{r_2}}{p_j}\right\rfloor\log p_j
 \ge  p_{r_2} \sum_{j=1}^{r_2} \frac{\log p_j}{p_j} - \sum_{j=1}^{r_2} \log p_j
 \ge   \frac{1}{3}\sqrt{\log x}\log\log x.
$$
Using this in (\ref{aux4}), we obtain
$$
\#\mathcal{B}_3(x)\le \#\left\{(c_1,\ldots , c_{r_3-r_2}):\sum_{j=1}^{r_3-r_2}c_j\le 3\frac{\sqrt{\log x}}{\log\log x}\right\},
$$
from which we can deduce that
\begin{equation}\label{aux5}
\#\mathcal{B}_3(x)\le \binom{\lceil\frac{3\sqrt{\log x}}{\log\log x}\rceil+r_3}{\lceil 3\frac{\sqrt{\log x}}{\log\log x}\rceil}
\le  \binom{\lceil\frac{3\sqrt{\log x}}{\log\log x}\rceil+\lceil 2\sqrt{\log x}\rceil }{\lceil 3\frac{\sqrt{\log x}}{\log\log x}\rceil}.
\end{equation}
Again using Lemma \ref{lem:binomial}, we conclude from (\ref{aux5}) that
\begin{equation}\label{b3}
\#\mathcal{B}_3(x)\le\exp\left\{(3+\varepsilon)\frac{\sqrt{\log x}\log\log\log x}{\log\log x}\right\}.
\end{equation}

We finally provide an upper bound for $\#\mathcal{B}_4(x)$ again using the same approach.  We have
$$
\#\mathcal{B}_4(x)\le\#\left\{(b_1,\ldots,b_{r-r_3}):\left(\sum_{j=1}^{r-r_3}b_j\log p_{j+r_3}\right)+\left(b_1\sum_{j=1}^{r_3}\left\lfloor\frac{p_{r_3}}{p_j}\right\rfloor\log p_j\right)\le \log x\right\}.
$$
Proceeding as before, we get
$$
\sum_{j=1}^{r_3}\left\lfloor\frac{p_{r_3}}{p_j}\right\rfloor \log p_j\ge \frac{1}{3}\sqrt{\log x}\,(\log\log x)^2,
$$
which yields
$$
\#\mathcal{B}_4(x)\le \#\left\{(c_1,\ldots, c_{r}):\sum_{j=1}^{r}c_j\le 3\frac{\sqrt{\log x}}{(\log\log x)^2}\right\},
$$
from which we conclude
$$
\#\mathcal{B}_4(x)\le \binom{\lceil\frac{3\sqrt{\log x}}{(\log\log x)^2}\rceil+r}{\lceil 3\frac{\sqrt{\log x}}{(\log\log x)^2}\rceil}
\le
\binom{\lceil\frac{3\sqrt{\log x}}{(\log\log x)^2}\rceil+\lceil 3\frac{\log x}{(\log\log x)^2}\rceil}{\lceil 3\frac{\sqrt{\log x}}{(\log\log x)^2}\rceil},
$$
so that
\begin{equation}\label{b4}
\#\mathcal{B}_4\le \exp\left\{(3+\varepsilon)\frac{\sqrt{\log x}}{\log\log x}\right\}.
\end{equation}
Gathering estimates (\ref{b1}), (\ref{b2}), (\ref{b3}) and (\ref{b4}) completes the proof of the  upper bound (\ref{eq:1}).

\section{The proof of the lower bound}

Many
elements of ${\cal J}$ have two or more representations as a product of factorials. For instance, the number $24=4!=2!^2\cdot 3!$ has two, whereas $576=4!^2 =2!^4\cdot 3!^2 = 2!^2\cdot 3!\cdot 4!$ has three. In fact, one can easily show that given an arbitrary integer $k\ge 2$, there exists a Jordan-P\'olya number which has $k$ representations as the product of factorials. For instance, take the numbers $n_k:=2^{3k+3} 3^{k+1}$ ($k=1,2,\ldots$).   One can easily check that
\begin{eqnarray*}
n_k & = & 4! \cdot 3!^k \cdot 2!^{2k} \\
    & = & 4!^2 \cdot 3!^{k-1} \cdot 2!^{2(k-1)} \\
    &   & \qquad  \vdots \\
    & = & 4!^{k-1} \cdot 3!^2 \cdot 2!^4 \\
    & = & 4!^k \cdot 3! \cdot 2!^2,
\end{eqnarray*}
thereby revealing  $k$ distinct representations of $n_k$ as a product of factorials.

This phenomenon must be taken into account when establishing a lower bound for ${\cal J}(x)$. This is why we will consider a  subset of ${\cal J}$ whose elements have a unique representation as a product of ``prime factorials''. We choose ${\cal J}_\wp$ as the set of those elements $n\in {\cal J}$ which can be written as a product of prime factorials, that is, as $n=\prod_{i=1}^r p_i!^{\alpha_i}$ for some non negative integers $\alpha_i$'s,  where $p_1,p_2,\ldots$  stands for the sequence of primes. The interesting feature of this set is that one can easily show that each of its elements has a unique representation as a product of prime factorials. Observe that ${\cal J}\setminus {\cal J}_\wp\neq \emptyset$ since it contains the number $n=14!$ and in fact many more.
\vskip 5pt

We will establish a lower bound for ${\cal J}_\wp(x)$, which will ipso facto provide a lower bound for ${\cal J}(x)$.
Given a large number $x$, let $z=\log x$ and set $a_i=\log(p_i!)$ for $i=1,2,\ldots,k$. Then, applying the first inequality in relation (\ref{eq:ennola}) of Proposition \ref{prop:ennola}, we get that, for each positive integer $k$,
\begin{equation} \label{eq:p0}
{\cal J}_\wp(x) > \frac{\log^k x}{k! \cdot \prod_{i=1}^k \log(p_i!)}.
\end{equation}
Let $\varepsilon>0$ and let $k_0$ be a large integer.
Using Lemma \ref{lem:stirling},
we may write that for each large prime $p_i$, say with $i\ge k_0$,
$$
\log(p_i!) =  p_i\log p_i -p_i + O( \log p_i) =p_i\log p_i \left( 1 - \frac 1{\log p_i} +O \left( \frac 1p_i \right) \right),
$$
so that, for each $k>k_0$, we have
\begin{equation} \label{eq:p1}
\prod_{i=k_0}^k \log (p_i!) = \prod_{i=k_0}^k p_i \cdot \prod_{i=k_0}^k \log p_i \cdot \prod_{i=k_0}^k \left( 1 -\frac 1{\log p_i} +O\left( \frac 1{p_i}\right) \right).
\end{equation}
We will now overestimate each of the above three products.

Using Proposition \ref{prop:pnt}, we have
\begin{eqnarray} \label{eq:p2} \nonumber
\prod_{i=k_0}^k p_i & < & \exp\left\{ \sum_{p\le p_k} \log p  \right\} =\exp\left\{ \theta(p_k)  \right\} = \exp\left\{ p_k \left( 1 + O \left( \frac 1{\log^2 k} \right) \right) \right\} \\
& < & \exp\left\{ ( k\log k + k\log \log k)  \left( 1 + O \left( \frac 1{\log^2 k} \right) \right)   \right\},
\end{eqnarray}
where we used inequality (\ref{eq:pk-6}) of Lemma \ref{lem:rosser}.

On the other hand, using once more the first inequality in Lemma \ref{lem:rosser}, we easily observe that
$\log \log p_i < (1+\varepsilon) \log \log i$ provided $i$ is sufficiently large.
It follows that
\begin{equation} \label{eq:p3}
\prod_{i=k_0}^k \log p_i = \exp\left\{ \sum_{i=k_0}^k \log \log p_i \right\} < \exp\left\{ \sum_{i=k_0}^k (1+\varepsilon) \log \log k  \right\}
< \exp\{ (1+\varepsilon) k \log \log k \}.
\end{equation}

Finally,
\begin{eqnarray} \label{eq:p4} \nonumber
\prod_{i=k_0}^k \left( 1 - \frac 1{\log p_i} +O\left( \frac 1{p_i} \right) \right)
& = & \exp\left\{ \sum_{i=k_0}^k \log \left( 1 - \frac 1{\log p_i} +O\left( \frac 1{p_i} \right) \right) \right\} \\
& = &  \exp\left\{ - \sum_{i=k_0}^k \frac 1{\log p_i} + O\left( \sum_{i=k_0}^k \frac 1{\log^2 p_i} \right) \right\} .
\end{eqnarray}
Since
\begin{eqnarray} \label{eq:p5} \nonumber
\sum_{i=k_0}^k \frac 1{\log p_i} & = & \int_{p_{k_0}}^{p_k} \frac 1{\log t} d\,\pi(t) =\left. \frac{\pi(t)}{\log t} \right|_{p_{k_0}}^{p_k}
+ \int_{p_{k_0}}^{p_k} \frac{\pi(t)}{t\log^2 t}\,dt \\
& = & \frac k{\log p_k} + O \left( \frac k{\log^2 p_k} \right) > \frac k{\log k} + O\left( \frac {k\log \log k}{\log^ 2 k} \right),
\end{eqnarray}
using estimate (\ref{eq:p5}) in (\ref{eq:p4}), we find that
\begin{equation} \label{eq:p6a}
\prod_{i=k_0}^k \left( 1 - \frac 1{\log p_i} +O\left( \frac 1{p_i} \right) \right) < \exp\left\{ - \frac k{\log k} + O\left( \frac {k\log \log k}{\log^ 2 k} \right) \right\}.
\end{equation}

Setting $\displaystyle{C_0=\prod_{i=1}^{k_0-1} \log(p_i!)}$ and gathering inequalities (\ref{eq:p2}), (\ref{eq:p3}) and (\ref{eq:p6a}) in (\ref{eq:p1}), we find that
\begin{equation} \label{eq:p7}
\prod_{i=1}^k \log(p_i!) = \prod_{i=1}^{k_0-1} \log(p_i!) \cdot \prod_{i=k_0}^k \log(p_i!) < C_0 (k \log ^2 k)^k \ e^{-k/\log k}.
\end{equation}
Finally, using Lemma \ref{lem:stirling}, we have that, provided $k_0$ is large enough,
\begin{equation} \label{eq:p8}
k! < (1+\varepsilon) k^k\, e^{-k} \, \sqrt{2\pi k} \qquad (k\ge k_0).
\end{equation}
Combining (\ref{eq:p7}) and (\ref{eq:p8}) in (\ref{eq:p0}), we obtain that
\begin{equation} \label{eq:p9}
{\cal J}_\wp(x) > \frac 1{C_0} \left( \frac{e^{1+1/\log k}\log x}{(1+\varepsilon) k^2 \log^2 k } \right)^k \qquad (k\ge k_0).
\end{equation}
Our goal will be to search for an integer $k=k(x)$ for which the function
$$ f(k):= \left( \frac{e^{1+1/\log k}\log x}{(1+\varepsilon) k^2 \log^2 k } \right)^k $$
reaches its maximum value, or equivalently for which real number $s$ the function $g(s):=\log f(s)$ reaches its maximum value.
Since
$$g(s)=s\left( 1+ \frac 1{\log s} +\log \log x -2\log s - 2\log \log s \right),$$
we have
\begin{eqnarray*}
g'(s) & = & 1+ \frac 1{\log s} +\log \log x -2\log s - 2\log \log s + s \left(- \frac 1{s\log ^2 s}  - \frac 2s - \frac 2{s\log s} \right) \\
& = &-1 + \log \log x -2\log s - 2\log \log s  - \frac 2{\log s}.
\end{eqnarray*}
For large $x$ and large $s$, the right hand side of the above expression will be near 0 when
$$\log \log x -2\log s - 2\log \log s =0,$$
or similarly,
$\log(s^2\log ^2s) = \log \log x$ and therefore  $s\log s = \sqrt{\log x}$, from which we find that
$$ s = \frac{\sqrt{\log x}}{\log s } = (1+o(1)) \frac{2\sqrt{\log x}}{\log \log x} \qquad (x\to \infty).$$
Substituting this value of $s$ in (\ref{eq:p9}), we find that
$$
{\cal J}_\wp(x) >  \exp\left\{ (1+o(1)) \frac{2\sqrt{\log x}}{\log \log x} \right\} \qquad (x\to \infty),
$$
thus establishing the required lower bound (\ref{eq:1a}).
\vskip 20pt

\noindent
{\bf Acknowledgements} The work of the first and second authors was supported by individual discovery grants from NSERC of Canada.

\vfill
\noindent

\end{document}